\documentclass[11pt]{amsart}

\usepackage{amsmath, amssymb, amscd, amsthm, amsfonts}
\usepackage{amsmath}
\usepackage{graphicx}

\usepackage{mathrsfs}
\usepackage{hyperref}
\DeclareMathOperator{\tr}{ tr}
\usepackage{tikz-cd} 

\usepackage{hyperref}
\hypersetup{colorlinks,linkcolor={blue},citecolor={blue},urlcolor={red}}

\title{Non-ergodicity on the $\SU(2)$-character varieties}

\author{Fayssal Saadi}

\email{fayssal.saadi@univ-tlse3.fr}

\address{Institut de Mathématiques de Toulouse, Université Paul Sabatier Toulouse 3, and UMPA, ENS de Lyon}  

\date{\today}

\newtheorem{theorem}{Theorem}[section]

\newtheorem*{theorem*}{Theorem}

\theoremstyle{definition}

\newtheorem{definition}{Definition}[section]

\newtheorem{ex}{Example}[section]

\newtheorem{prop}{Proposition}

\newtheorem{lemma}{Lemma}

\newtheorem{remark}{Remark}[section]

\newtheorem{Cor}{Corollary}[section]

\def\Aut{\sf{Aut}}

\def\SO{{\sf{SO}}}

\def\SU{{\sf{SU}}}

\def\SL{{\sf{SL}}}

 \def\Hom{{\sf{Hom}}}

\def\tr{{\sf{tr}}}

\begin{document}

\maketitle
\maketitle

\begin{abstract}
 We describe the dynamics of a group $\Gamma$ generated by Dehn twists along two filling multi-curves or a family of filling curves on the $\SU(2)$-representation variety of closed surfaces. Consequently,  we provide explicit $\Gamma$-invariant rational functions on the representation variety of the genus two closed surface $S_2$ for some pair of multi-curves. We establish a similar result for the $\SU(2)$-character variety of genus four non-orientable surfaces $N_4$ for some family of filling curves.    
\end{abstract}
\section{Introduction}
For a closed surface $\Sigma$, we define the $\SU(2)$-representation variety of $\Sigma$, denoted by $\Hom(\pi_1(\Sigma),\SU(2))$, to be the space of all homomorphisms of the fundamental group of $\Sigma$ into the special unitary group $\SU(2)$. The group $\SU(2)$ acts by conjugacy on $\Hom(\pi_1(\Sigma),\SU(2))$ defining a principal $\SU(2)$-bundle over the irreducible  representations of the character variety: 

$$X(\pi_1(\Sigma),\SU(2)):=  \Hom(\pi_1(\Sigma),\SU(2)) / \SU(2)$$

The group of homeomorphisms of $\Sigma$ acts on the representation variety by pre-composition and the action of elements isotopic to the identity is absorbed by the $\SU(2)$ action, giving rise to a natural action of the mapping class group of $\Sigma$ on  $X(\pi_1(\Sigma),\SU(2))$.

For an orientable surface $S$, Goldman \cite{G1} showed that the character variety inherits a symplectic form $\omega$ and proved that the mapping class group acts ergodically with respect to the induced measure \cite{GX1}. A natural question is then to ask whether a subgroup $\Gamma$ of the mapping class group acts ergodically or not. The first result in this direction is given by Goldman and Xia \cite{GX2} by proving that on the twice-punctured torus, the Torelli group acts ergodically on its character variety. In \cite{FM}, Funar and Marché showed that the first Johnson subgroup, which is the group generated by Dehn-twists along separating curves of $\Sigma$, acts ergodically on the character variety. Recently, Marché and Wolff \cite{MW} proved that any non-central normal subgroup of the mapping class group acts topologically 
transitively on its $\SU(2)$-character variety.

Brown \cite{Br} on the other hand proved that a pseudo-Anosov element (i.e. the iterates of such an element preserve no essential simple closed curve) admits an elliptic fixed point or a double elliptic fixed point for some relative character varieties of the punctured torus $S_{1,1}$. This recently led to fully proving the non-ergodicity of such elements by applying KAM theory in \cite{FGLM} by Forni, Goldman, Lawton, and Matheus. 

In parallel, the ergodicity of the mapping class group of non-orientable surfaces was proved by Palesi \cite{P} after introducing a measure $\upsilon$ invariant by the mapping class group.

In this article, we are interested in subgroups $\Gamma$ generated by Dehn twists along a pair of filling multi-curves or along a family of filling curves. Due to the non-commutativity of $\SU(2)$ and the fact that 
the curves fill the surface, it seems that the natural answer would be the ergodicity of $\Gamma$ on the representation variety. Surprisingly, we find in the non-orientable setting:

\begin{theorem}\label{Th1}
On the closed non-orientable surface of genus four $N_4$, there exists a family of filling curves whose associated Dehn twists generate a group $\Gamma$ acting non-ergodically on the character variety $X(\pi_1(N_4),\SU(2))$.
\end{theorem}

The application of Theorem 6.1 by Fathi \cite{F} ensures the existence of pseudo-Anosov elements in such a group. Therefore:  

\begin{Cor}
There exists a pseudo-Anosov element on $N_4$ which does not act ergodically on the character variety $X(\pi_1(N_4),\SU(2))$. 
\end{Cor}

On the orientable surface of genus 2, we can show that $\Aut^{+}(\pi_1(S_2))$ acts ergodically on $\Hom(\pi_1(S_2),\SU(2))$ (This will be the content of the note \cite{S}). But for a suitable $\Gamma$, we establish:   

\begin{theorem}\label{Th2}
Let $S_2$ be the orientable surface of genus two. Then, there exist two filling multi-curves whose associated Dehn twists generate a group $\Gamma$ acting non-ergodically on $\Hom(\pi_1(S_2),\SU(2))$. This $\Gamma$ could be chosen to contain a finite index subgroup of the Veech group of a square-tiled surface.
\end{theorem}

\subsection*{Organization of the paper:} 
\begin{itemize}
    \item In Section \ref{Square-tiled S}, we introduce the notion of square-tiled surfaces as a way to represent pairs of filling multi-curves. This provide a way to construct pseudo-Anosov homeomorphisms and to realize some elements of $\Gamma$ as affine transformations. 
    \item Section \ref{Goldman Flow} is devoted to describing the action of a single Dehn-twist defining Goldman's flows. 
    \item Section \ref{Foliations} is somehow a generalization of the action of Dehn-twists along two filling multi-curves to natural flows on algebraic varieties, in an attempt to read the dynamics using the ideal defining the algebraic variety. 
    \item In the last Section \ref{Examples}, we prove Theorems \ref{Th2} and \ref{Th1} by providing examples of $\Gamma$'s admitting explicit invariant rational functions. 
\end{itemize}
The search for rational functions follows from the description in Section \ref{Foliations}. However, it is not possible to aim for polynomial invariant functions. Indeed, it was pointed out to me after a discussion with J. Marché based on Theorem 1.1 in \cite{CM} that no pseudo-Anosov element on orientable surfaces admits an invariant polynomial function on the $\SU(2)$-character variety. 

\subsection*{Acknowledgement}

I would like to warmly thank my advisors Louis Funar, Erwan Lanneau, and Abdelghani Zeghib for their continuous support. I also would like to thank Julien Marché and Maxime Wolff for their helpful remarks.   

This work was supported by the LABEX MILYON (ANR-10-LABX-0070) of Université
de Lyon, within the program “Investissements d’Avenir” (ANR-11-IDEX-0007) operated
by the French National Research Agency (ANR).

\section{multi-twists and square-tiled surfaces}\label{Square-tiled S}
One way of generating pseudo-Anosov homeomorphisms on a surface $S$ is by considering the group generated by Dehn-twists along two filling multi-curves. 
\begin{definition}
A square-tiled surface is a finite collection of squares on $\mathbb{C}$, where edges are glued together two by two via a translation or a half-translation (i.e. a similarity with linear part $-1$).
\end{definition} 
The square-tiled surfaces are naturally endowed with a half-translation structure i.e. a structure where the transition maps are of the form $z \mapsto \pm z+c$, for some $c \in \mathbb{C}$. An affine transformation of $S$ is then a transformation such that it is affine for the charts of the half-translation structure. The group generated by all the linear parts of such transformations is called the Veech group of $S$. We shall point out here that $\SL(2,\mathbb{R})$ (in particular $\SL(2,\mathbb{Z})$) acts on the set of square-tiled surfaces by post-composition i.e. if $S$ is a half-translation surface and $A \in \SL(2,\mathbb{R})$ then $A.S$ is the half-translation surface obtained by composing the charts of $S$ with $A$. In particular, we have that the stabilizer of a half-translation surface is its Veech group. The following then holds for square-tiled surfaces (See \cite{GJ}): 
\begin{prop}
The Veech group of a square-tiled surface is a finite index subgroup of $\SL(2,\mathbb{Z})$.
\end{prop}

From a square-tiled surface, one can obtain two multi-curves that consist of horizontal curves $\gamma_1,\dots,\gamma_n$ and vertical ones $\lambda_1,\dots,\lambda_m$. In fact the surface $S$ is decomposed in two different ways via vertical and horizontal cylinders and the curves $\gamma_1,\dots,\gamma_n$ (resp. $\lambda_1,\dots,\lambda_m$) are exactly the generators of the fundamental groups of the horizontal cylinders (resp. vertical cylinders). 

The converse is also true, let $\gamma=\{\gamma_1,\dots,\gamma_n \}$ and $\lambda= \{\lambda_1,\dots,\lambda_m\}$
be two filling multi-curves on an orientable surface $S$ that is their complementary set is a union of disks. Then we can construct a square-tiled surface by simply considering a square on each intersection between the curves centered at the intersection point, and the gluing is deduced according to the combinatorial data of $\gamma$ and $\lambda$. We conclude:

\begin{prop}
On orientable surfaces, the set of two filling multi-curves is in bijection with the set of square-tiled surfaces, up to homeomorphisms.    
\end{prop}

In the case where the intersections between the curves of $\gamma$ and $\lambda$ have the same sign with respect to the orientation of $S$, then the corresponding square-tiled surface has only translations as identifications between the edges of the squares, hence we get an origami i.e. a ramified cover of the torus with only one singular point at the basis (in this case the torus). To encode the origami we need two permutations: let $S_1,\dots, S_d$ be the squares forming the origami $S$ (Notice here that $d$ is the degree of the ramified cover). For $i \in \{1,\dots d\}$ denote by $a_i$ the edge to the left of the square $S_i$ and by $b_i$ the edge to the bottom of $S_i$. To determine the origami it is sufficient to decide which edge $a_{\sigma_{i}}$ is to the right to the square $S_i$ and which edge $b_{\sigma'_{i}}$ is to the top of the square $S_i$. Therefore one has: 
\begin{prop}
The set of pair of permutations that acts transitively on $\{1,\dots,d \}$, up to conjugation of the permutation group of $d$ elements, is in bijection with the set of connected origamis of degree $d$, up to homeomorphisms. 
\end{prop}
For instance, consider $d=4$, $\sigma=(1,2,3,4)$ and $\sigma'=(1,3,2,4)$. In this example, we get a genus $2$ surface where the corresponding two multi-curves $\gamma$ and $\lambda$ are exactly two curves, the vertical curve $\gamma$ and the horizontal curve $\lambda$, as illustrated in Figure \ref{2 curves}.  

\begin{figure}[htp]
    \centering
    \includegraphics[width=7cm]{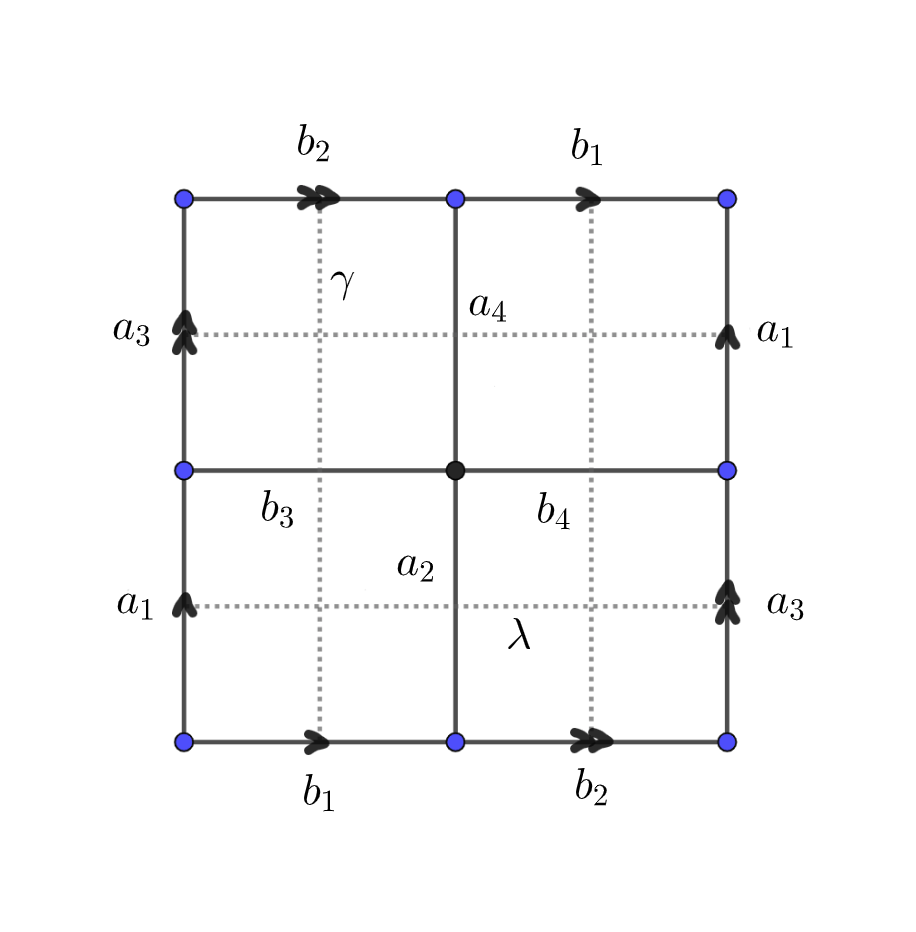}
   \caption{Two filling curves in $S_2$}
    
    \label{2 curves}
\end{figure}

The vertical curve $\gamma$ (resp. $\lambda$) generates the fundamental group of the vertical (resp. horizontal) cylinder made by the four squares. Viewing the surface in this way allows one to see that the matrix $\begin{pmatrix}
1 & 0\\
4 & 1 
\end{pmatrix}$  (resp. $\begin{pmatrix}
1 & 4\\
0 & 1 
\end{pmatrix}$) realizes an affine transformation on the surface that is in the isotopy class of the Dehn-twist along the curve $\gamma$ (resp. $\lambda$). In particular, the vertices of the square-tiled surface are fixed point under the action of the affine transformations corresponding to Dehn-twists along the horizontal and vertical curves.

\section{ $\SU(2)$-character varieties and Goldman's flow}\label{Goldman Flow}

The simplest actions on $\Hom(\pi_1(S),\SU(2))$ are those defined via a Dehn-twist along a simple closed curve. If a curve $\alpha$ is non-separating, then its action on $\Hom(\pi_1(S),\SU(2))$ can be expressed in a simple way using the HNN extension by writing: $$\pi_{1}(S) =  \{\pi_{1}(S \setminus \alpha) ,  \beta  \mid   \beta \alpha^{+} \beta^{-1}=\alpha^{-} \}$$ 

With $\alpha^{+}$ and $\alpha^{-}$ being the resulting boundaries loops after cutting along $\alpha$. The group $\pi_1(S \setminus \alpha)$ can be seen as the group that consists of curves that do not intersect $\alpha$. The loop $\beta$ is a loop that intersects $\alpha$ in one point. If we consider the base point of the fundamental group of $S$ to be on the boundary of a tubular neighborhood of $\alpha$,  then, on the $\Hom$ level, we can write: 

\begin{equation*}
\tau_{\alpha}^{*}(\rho)(\gamma)=
    \begin{cases}
        \rho(\gamma) & \text{if } \gamma \in \pi_{1}(S \setminus \alpha)\\
        \rho(\gamma) . \rho(\alpha^{+}) & \text{if } \gamma =\beta 
    \end{cases}
\end{equation*}

If we denote the one-parameter subgroup of $\SU(2)$ of velocity $1$ that passes through $\rho(\alpha^{+})$ (choosing the shortest path between $Id$ and $\rho(\alpha^{+})$) by $t \mapsto \xi_{\rho(\alpha^{+})}(t)$), with respect to the bi-invariant metric on $\SU(2)$, then we have a $2\pi$-periodic flow  $\Xi_{\alpha}$ of $\mathbb{S}^{1}$ on $\Hom(\pi_1(S),\SU(2))$:  

\begin{equation*}
\Xi_{\alpha}^{t}(\rho)(\gamma)=
    \begin{cases}
        \rho(\gamma) & \text{if } \gamma \in \pi_{1}(S \setminus \alpha)\\
        \rho(\gamma) . \xi_{\rho(\alpha^{+})}(t) & \text{if } \gamma =\beta 
    \end{cases}
\end{equation*}
It is a well-defined flow on the character variety except when $\tr(\rho(\alpha^{+}))=\pm 2$. We restate here a version of Theorem 4.3 in \cite{G2} in the case where $\alpha$ is a non-separating curve:

\begin{prop}
 The Hamiltonian flow of the function $\rho \mapsto \tr(\rho(\alpha))$ with respect to the symplectic form is a reparametrization of the previous flow $\Xi_{\alpha}$.   
\end{prop}

Using the HNN extension, the restriction to $\pi_1(S \setminus \alpha)$ defines a projection $P_{\alpha} : \Hom(\pi_1(S),\SU(2)) \mapsto \Hom(\pi_1(S \setminus \alpha),\SU(2))$.  Due to the above description, one has:  

\begin{prop}\label{Pro4}
The projection $P_{\alpha}$ defines an $\mathbb{S}^1$-bundle on the locus $\{ \rho \in \Hom(\pi_1(S),\SU(2)) \mid \rho(\alpha^{+}) \neq \pm Id)$ over its image. Moreover, its fibers coincide with the orbits of $\Xi_{\alpha}$.      
\end{prop}

\begin{proof}
We notice first that the projection $P_{\alpha}$ is invariant under the flow $\Xi_{\alpha}$. Now we need to prove that the orbits of $\Xi_{\alpha}$ are exactly the fibers of $P_{\alpha}$.

Let $\rho_{0}$ be a representation in $\Hom(\pi_1(S \setminus \alpha),\SU(2))$, We observe that the image of $P_{\alpha}$ is the algebraic subset defined by the polynomial function $\rho \mapsto \tr(\rho(\alpha^{+}))- \tr(\rho(\alpha^{-}))$. Therefore we assume that $\tr(\rho_{0}(\alpha^{+}))=\tr(\rho_{0}(\alpha^{-}))$. In order to find an extension of $\rho_{0}$ in $\Hom(\pi_1(S),\SU(2))$, it is sufficient to determine $\rho(\beta)$. The condition $\beta.\alpha^{+}.\beta^{-1}=\alpha^{-}$ implies that $\rho(\beta)$ lies in a big circle $S_{\alpha} \subset \SU(2)$ that depends only on $\rho_{0}(\alpha^{+})$ and $\rho_{0}(\alpha^{-})$. 

\end{proof}

The situation is not very different for a non-separating multi-curve $\alpha = \{\alpha_1,\dots,\alpha_n \}$. In fact the flows $\Xi_{\alpha_1},\dots,\Xi_{\alpha_n}$ commutes, therefore we have an action of an n-dimensional torus $\mathbb{T}^n$ defined almost everywhere on the representation variety. 

Similarly, we denote by $P_{\alpha}$ the restriction to $\pi_1(S \setminus \alpha)$. Let us now denote by $M_{\alpha}$ the group generated by $\tau_{\alpha_{1}},\dots,\tau_{\alpha_{n}}$ and consider a homeomorphism $f=\tau_{\alpha_{k_1}}^{i_1} \circ \dots \circ \tau_{\alpha_{k_n}}^{i_n}$ in $M_{\alpha}$, for some non-zero integers $i_1,\dots,i_n$, then one has:   

\begin{prop}\label{P6}
   The ergodic components of $f$, the ergodic components of $M_{\alpha}$, and the fibers of $P_{\alpha}$ are almost everywhere equal. 
\end{prop}

\begin{proof}

The proof relies on the relation between the action of a single Dehn-twist $\tau_{\gamma}$ and the flow $\Xi_{\gamma}$, one has the following  (See Section 2 in \cite{GX1} for more details):   
$$  \tau^{*}_{\gamma}(\rho) =  \Xi_{\gamma}^{\theta(\rho(\gamma^{+}))}(\rho) $$ 

Where $\theta(X)$ is the angle of the matrix $X$, more precisely, $\theta(X)=\arccos(\frac{\tr(X)}{2})$.
Since the curves do not intersect, $f$ preserves the functions $\rho \mapsto \theta(\rho(\alpha_{i}^{+}))$, for all $i \in \{1,\dots,n\}$. Hence we deduce that for a generic representation $\rho$, the automorphism $f$ acts by the same translation inside the $\Xi_{\alpha}$-orbit of $\rho$ which is generically diffeomorphic to a torus $\mathbb{T}^{n}$. The fact that $\{\theta(\rho(\alpha_{i}^{+})\}_{i=1}^{i=n}$ can be chosen freely yields the first part. The proof of the second part i.e. the relation between $P_{\alpha}$ and $M_{\alpha}$ is the same as the proof of Proposition \ref{Pro4}.

\end{proof}

The last proposition says that one can not distinguish measurably between
the action of such an $f \in M_{\gamma}$ and the action of the whole group $M_{\gamma}$.

\begin{remark}
   If $\alpha = \{ \alpha_i \}_{i=1}^{i=n}$ is a non-separating multi-curve then the image of $P_{\alpha}$ is the algebraic variety defined to be the zero locus in $\Hom(\pi_1(S\setminus \alpha),\SU(2))$ of the polynomials $\rho \mapsto \tr(\rho(\alpha_{i}^{+}))-\tr(\rho(\alpha_{i}^{-}))$, for $i \in \{1,\dots n \}$.

\end{remark}

\section{Foliations on the intersection of quadrics}\label{Foliations}

To analyze the action of two multi-twists, let us reflect on a larger family of algebraic foliations. 
On the affine space $\mathbb{R}^n \times \mathbb{R}^n$, let  $\{f_i\}_{i=1}^{i=m}$ be a collection of bilinear forms on  $\mathbb{R}^n \times \mathbb{R}^n$ and denote by $f$ the bilinear map  $(f_1,\dots,f_m)$. 
Set $V$ to be the zero locus of the map $f$. Over the algebraic variety $V$ we derive two natural foliations:

For $A \in \mathbb{R}^n$ and $B \in \mathbb{R}^n$, let $\mathcal{A}$ and $\mathcal{B}$ be the intersection of $V$ with the affine subspaces $\{ A\} \times \mathbb{R}^n$ and $\mathbb{R}^n \times \{ B\} $ i.e. the levels of the maps $X \mapsto f(A,X)$ and $X \mapsto f(X,B)$, respectively. Let $\mathcal{C}$ to be the equivalence relation generated by the two foliations.  
Let $P_A: V \mapsto \mathbb{R}^n$ and  $P_B: V \mapsto \mathbb{R}^n$ be the projection to the first and the second components, respectively.

At this point, one can ask whether the saturation $\mathcal{C}$ fills the variety $V$ or not. The answer depend on the ability to separate the two variables $A$ and $B$ from each other, in other words:

\begin{prop}
If $V$ admits a non-constant function that factors through $P_A$ and $P_B$ simultaneously, then the function is constant on the $\mathcal{C}$-classes. 
\end{prop}

\begin{ex}
  Let $V \subset \mathbb{R}^2 \times \mathbb{R}^2 $ be the quadric defined by the bilinear map $ f(a_1,a_2,b_1,b_2) = a_1b_1 + a_2b_2$. Then the function $\frac{a_1}{a_2} = - \frac{b_2}{b_1}$ is constant on the $\mathcal{C}$-classes.  
\end{ex}

\begin{ex}
 If we consider $\mathbb{R}^3$ instead of $\mathbb{R}^2$ that is the quadric $V \subset \mathbb{R}^3 \times \mathbb{R}^3 $ defined by $f(a_1,a_2,a_3,b_1,b_2,b_3) = a_1b_1 + a_2b_2+a_3b_3$, then no such function exists. We simplify the task by considering the zero locus of $a_1b_1+a_2b_2 = -1$ on the product of the projective spaces of the two components. Therefore the new foliations $\mathcal{A}$ and $\mathcal{B}$ are the integral curves of the vector fields $(0,0,b_2,-b_1)$ and $(a_2,-a_1,0,0)$, respectively. 
\end{ex}

\subsection*{On the $\SU(2)$-representation variety}

Let $\gamma$ and $\lambda$ be two filling multi-twists with one disk as complementary set and $n$ positive intersections in total (intersections having the same sign with respect to the orientation of $S$). The dynamics of the group generated by $M_\gamma$ and $M_{\lambda}$ on the representation variety $\Hom(\pi_1(S),\SU(2))$ ties into the previous discussion.  Let $S$ be the origami associated to $\gamma$ and $\lambda$. Let $\{ a_1,\dots ,a_n \}$ and  $\{ b_1,\dots ,b_n \}$ be the generators of $\pi_1(S \setminus \gamma)$ and $\pi_1(S \setminus \lambda)$, respectively.  

\begin{remark}
    Here we consider the vertex of the origami (conical singularity) to be the base point of the fundamental group. In this way, we can apply the results of Section \ref{Goldman Flow}.
\end{remark}

\begin{figure}[htp]
    \centering
    \includegraphics[width=6cm]{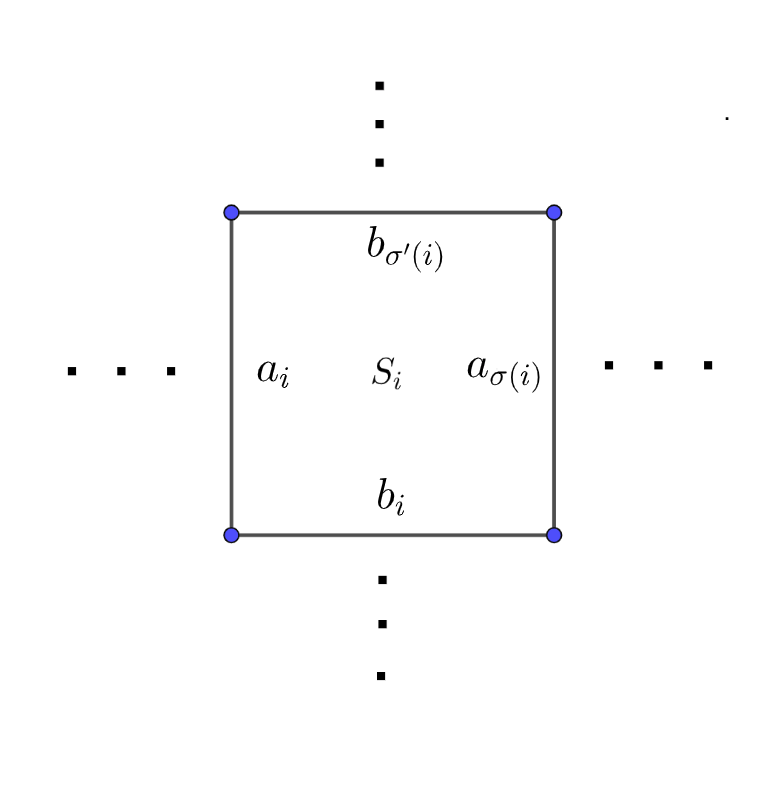}
   \caption{Square $S_i$ in the Square-tiled surface}
    
    \label{Square}
\end{figure}

The fundamental groups $\pi_1(S \setminus \gamma)$, $\pi_1(S \setminus \lambda)$ and $\pi_1(S)$ are generated
(with redundancy) by $\{ a_1,\dots, a_n \}$, $\{ b_1,\dots, b_n\}$ and  $\{ a_1,b_1,\dots,a_n,b_n \}$, respectively. 
The relations between the generators  $\{ a_1,\dots ,a_n \}$ and  $\{ b_1,\dots ,b_n \}$ are given by the square relations: $$a_i.b_{\sigma'(i)} = b_{i}.a_{\sigma(i)} $$ for  $i \in \{1,\dots n\}$. For a representation $\rho$ set $A_{i}=\rho(a_i)$ and $B_i=\rho(b_i)$, for $i \in \{1,\dots,n\}$. Using Proposition \ref{P6}, we get:  

\begin{prop}

    The map $\rho \mapsto (A_1,\dots ,A_n,B_1,\dots,B_n)$ is an embedding of $\Hom(\pi_1(S),\SU(2))$ into $\SU(2)^{2n}$ and the image $V$ is an algebraic variety defined by the square relations:  $$A_iB_{\sigma'(i)}-B_{i}A_{\sigma(i)}=0$$
    
    In addition, the foliations $\mathcal{A}$ and $\mathcal{B}$ defined on $V \subset \SU(2)^n \times \SU(2)^n$ are exactly the ergodic component of the subgroups $M_{\gamma}$ and $M_{\lambda}$.

\end{prop}

\section{Invariant functions}\label{Examples}

In what follows, we view $\SU(2)$ as the unit sphere of the quaternion numbers $\mathbb{H}$. 
Denote by $\mathbb{H}^{0}$ the subspace of imaginary vectors. We endow $\mathbb{H}$ with the canonical scalar product $(X,Y) \mapsto \tr(X.\overline{Y})$ which is bi-invariant, i.e. if $A,B \in \SU(2)$, then the linear map:   
$$ X \mapsto  A.X.B $$
is an isometry of $\mathbb{H}$. Conversely, every isometry of $\mathbb{H}$ can be expressed in this way.

Let us consider the following specific rectangle made with one vertex, two horizontal edges $b_I$, $b_J$, and two vertical edges $a_I$ and $a_J$ as illustrated in Figure \ref{Rectangle} below.   

\begin{figure}[htp]
    \centering
    \includegraphics[width=5cm]{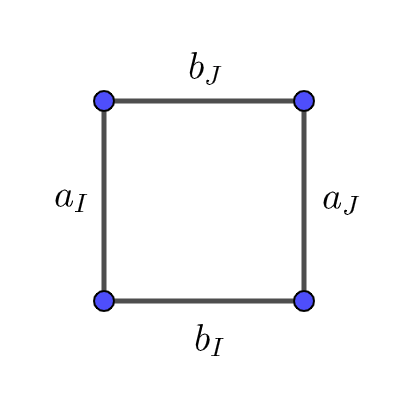}
    \caption{A rectangle in the square-tiled surface}
    \label{Rectangle}
\end{figure}
The rectangle relation writes $a_I.b_J =b_I.a_J$, so if $\rho$ is a representation, then we set $A_I=\rho(a_I)$, $A_J=\rho(a_J)$, $B_I=\rho(b_I)$ and $B_I=\rho(b_I)$. Now the rectangle relation writes:  $$ A_I.B_J=B_I.A_J$$
In general, it is not possible to separate the variables $(A_I,A_J)$ from $(B_I,B_J)$ i.e. to find a function that factors through $P_A$ and $P_B$ simultaneously. However, under some conditions the task becomes possible.

\begin{lemma}\label{One Square}
 If $a_I$ is conjugated to $a_J$ and $b_I$ is conjugated to $b_J$ then the two directions $[A_I-A_J]$ and $[B_I-B_J]$ defined on $P(\mathbb{H}^{0})$ are equal. 
\end{lemma}
    
\begin{proof}
Consider the following linear map, which is a function that depends only on  $A_I$ and $A_J$ :  

\begin{align*}
    \phi\colon&\mathbb{H}\longrightarrow\mathbb{H}\\
    &X \mapsto A_I.X.A_J^{-1}-X
     \end{align*}
The fact that $A_I$ is conjugated to $A_J$ implies that the kernel of $\phi$ is of rank $2$ since non-zero elements in the kernel are those who conjugate $A_I$ to $A_J$. So the image of $\phi$, denoted $Im(\phi)$, is of rank $2$. 
Observe that $Im(\phi)$ cannot contain only traceless matrices; if $\tr \circ \phi$ vanishes then, in particular, $\phi(1)=A_I.A_J^{-1}-1$  would be traceless which would imply that $A_I=A_J$ which does not hold in general.  
From the rectangle relation, we deduce that $$B_I-B_J = A_I.B_J.A_J^{-1}-B_J$$
Therefore $B_I-B_J \in Im(\phi)$. Since $B_I$ and $B_J$ are conjugate then $$ B_I-B_J = Im(\phi) \cap \mathbb{H}^{0}$$ The conjugacy assumption of $A_I$ and $A_J$ together with the fact that $\phi(A_J) = A_I-A_J$ allow us to deduce that: $$ [A_I-A_J]=[B_I-B_J]$$
\end{proof}

\subsection{Examples on the representation variety of $S_2$}(Proof of Theorem \ref{Th2})

In what follows, $\Gamma$ will denote the group generated by $M_{\gamma}$ and $M_{\lambda}$.
Using the previous lemma, one can see that the following examples admit a function on $\Hom(\pi_1(S),\SU(2))$ that factors through $P_A$ and $P_B$, simultaneously.   

For instance, let us consider the square-tiled surface $S'$ below (Figure \ref{Surface S'}) made of three squares with only one vertex and hence $S'$ has genus 2. 
\begin{figure}[htp]
    \centering
    \includegraphics[width=5.5cm]{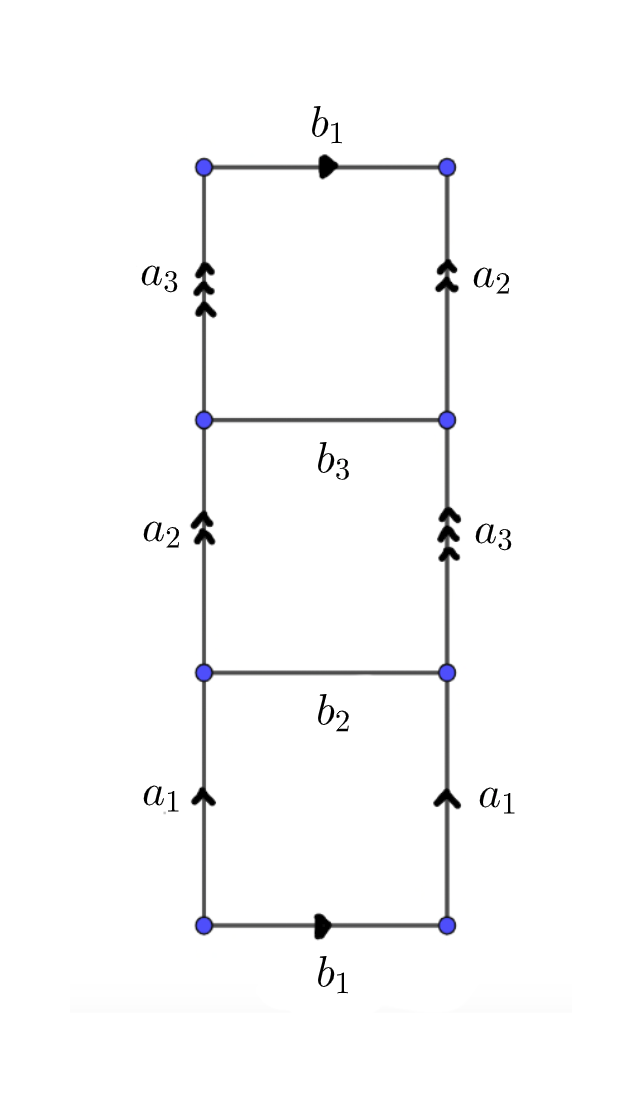}
    \caption{The surface $S'$}
    \label{Surface S'}
\end{figure}
We enumerate the squares of $S'$ from bottom to top. Using the relation of square $1$, i.e.  
$a_1.b_2=b_1.a_1$, we deduce that $b_1$ is conjugated to $b_2$. The curve  $a_2.a_3$ is always conjugated to $a_3.a_2$. By looking at the rectangle made by squares $2$ and $3$, using the relation $a_2.a_3.b_1=b_2.a_3.a_2$, we conclude that for a representation $\rho$, Lemma \ref{One Square}  applies and hence one has: 

$$[A_2.A_3-A_3.A_2] = [B_1 - B_2]$$
In other words, the direction orthogonal to $1$, $A_2$, and $A_3$ is exactly the direction of $B_1-B_2$. 

\begin{remark}
In this example, one can see that the affine transformations of $S'$ given by:  $$\begin{pmatrix}
1 & 0\\
3 & 1 
\end{pmatrix},   \begin{pmatrix}
1 & 2\\
0 & 1 
\end{pmatrix} $$  are elements of $\Gamma$.
\end{remark}

Lemma \ref{One Square} applies to different surfaces of genus $2$, For example, we can consider the famous square-tiled surface $S$ (Figure \ref{Famous Example} below). The group $\Gamma$ is generated by four Dehn-twists along two horizontal loops and two vertical ones as defined in Section $\ref{Square-tiled S}$. A finite index subgroup of the Veech group of $S$ is contained in $\Gamma$, this is a consequence of the fact that the two affine transformations of $S$:     
$$\begin{pmatrix}
1 & 2\\
0 & 1 
\end{pmatrix} ,  \begin{pmatrix}
1 & 0\\
2 & 1 
\end{pmatrix}$$ generate a finite index subgroup of $\SL(2,\mathbb{Z})$.

\begin{figure}[htp]
    \centering
    \includegraphics[width=6.2cm]{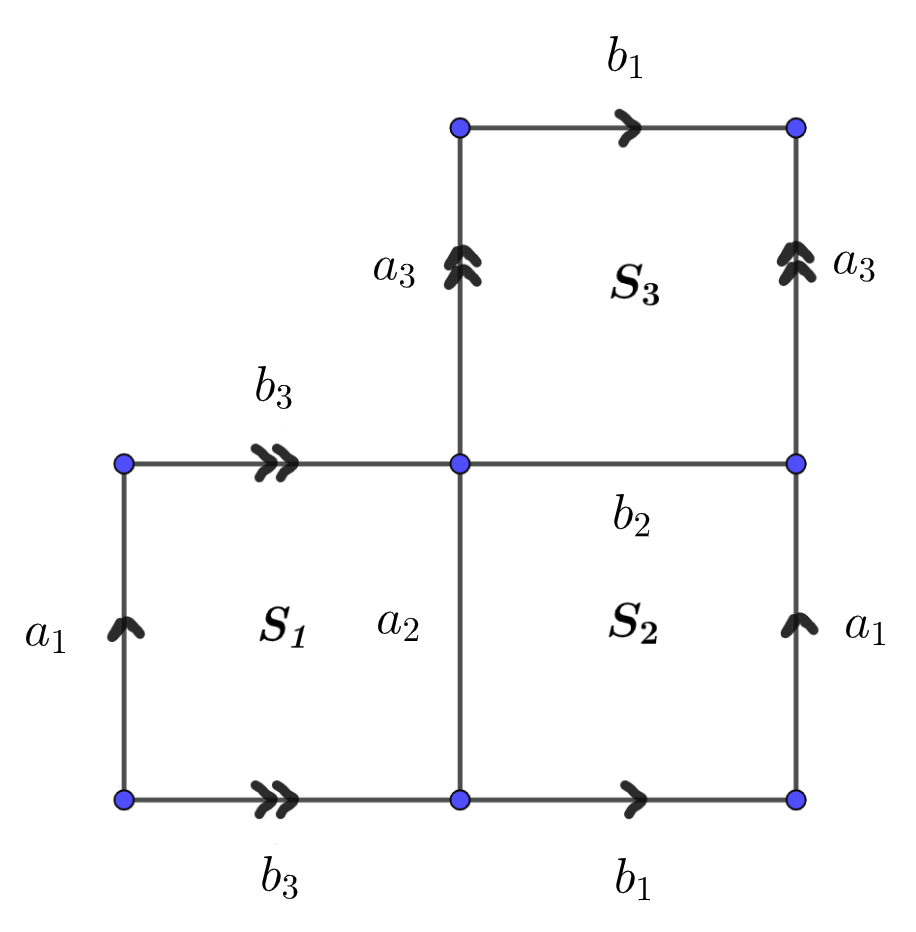}
   \caption{The genus two surface $S$}
    \label{Famous Example}
\end{figure}

After applying Lemma \ref{One Square} on the second square knowing from the first and the third squares that $a_1$ is conjugated to $a_2$ and $b_1$ is conjugated to $b_2$, we deduce that $[A_1-A_2]=[B_1-B_2]$, hence $\rho \mapsto [A_1-A_2]$ is $\Gamma$-invariant on $\Hom(\pi_1(S),\SU(2))$.

\subsection{Examples on the character variety of $N_4$}(Proof of Theorem \ref{Th1})

In this section, we will slightly modify the previous discussion by considering an arbitrary family of filling curves  $\gamma = \{ \gamma_1, \dots, \gamma_n \}$ on a surface $\Sigma$ (not necessarily consisting of two multi-curves).

Without loss of generality, one can assume that there are no three curves that intersect at a single point. Now one can reconstruct the closed surface using the combinatorial data of the curves $\gamma_1, \dots,\gamma_n$. Let $\gamma$ be an abstract family of curves, for each point of intersection $p$ consider a square centered at $p$ with edges transversal to the curves that perform the intersection at $p$. The gluing among the squares is then deduced by following the paths of the curves. As a result, if the surface is orientable, we get a $\frac{1}{4}$-translation surface i.e a Euclidean surface with identifications of the form  $z \mapsto i^{k}z+c$, for some $k \in \mathbb{Z}$ and $c \in \mathbb{C}$. To summarize:

\begin{prop}
The set of families of filling curves (possibly with self-intersections) is in bijection with the square-tiled $\frac{1}{4}$-translation surfaces, up to homeomorphisms.    
\end{prop}

The fact that any $\frac{1}{4}$-translation surface has a half-translation surface as a double ramified cover implies:     

\begin{Cor}
Any family of filling curves is the image of two filling multi-curves via a ramified cover of degree two.  
\end{Cor}

\begin{remark}
For a general closed surface $\Sigma$ (orientable or not), we shall add a reflection along the $x$-axis to the group of rotations and translations to get a structure corresponding to any family of curves on $\Sigma$. 
\end{remark}

On the $\SU(2)$-representation variety $\Hom(\pi_1(\Sigma),\SU(2))$. The group $\Gamma$ generated by the Dehn-twists along $\gamma_1,\dots, \gamma_n$ acts on it. Let us now denote the projection defined on Section \ref{Goldman Flow} associated with $\gamma_i$ by $P_i$. The previous discussion can be adapted to this situation: 

\begin{lemma}

A function that factors through $P_i$, for all $i \in \{1,\dots,n\}$, is a $\Gamma$-invariant function on $\Hom(\pi_1(\Sigma),\SU(2))$. 
    
\end{lemma}

The following example is a genus $4$ non-orientable surface denoted $N_4$. The curves $\{\gamma_1,\gamma_2,\gamma_3\}$ are in minimal position since the geometric intersection between any pair of curves is one, therefore the curves $\{\gamma_1,\gamma_2,\gamma_3\}$ are filling the surface. As shown in Figure \ref{N4} below:

\begin{figure}[htp]
    \centering

\includegraphics[width=7.20cm]{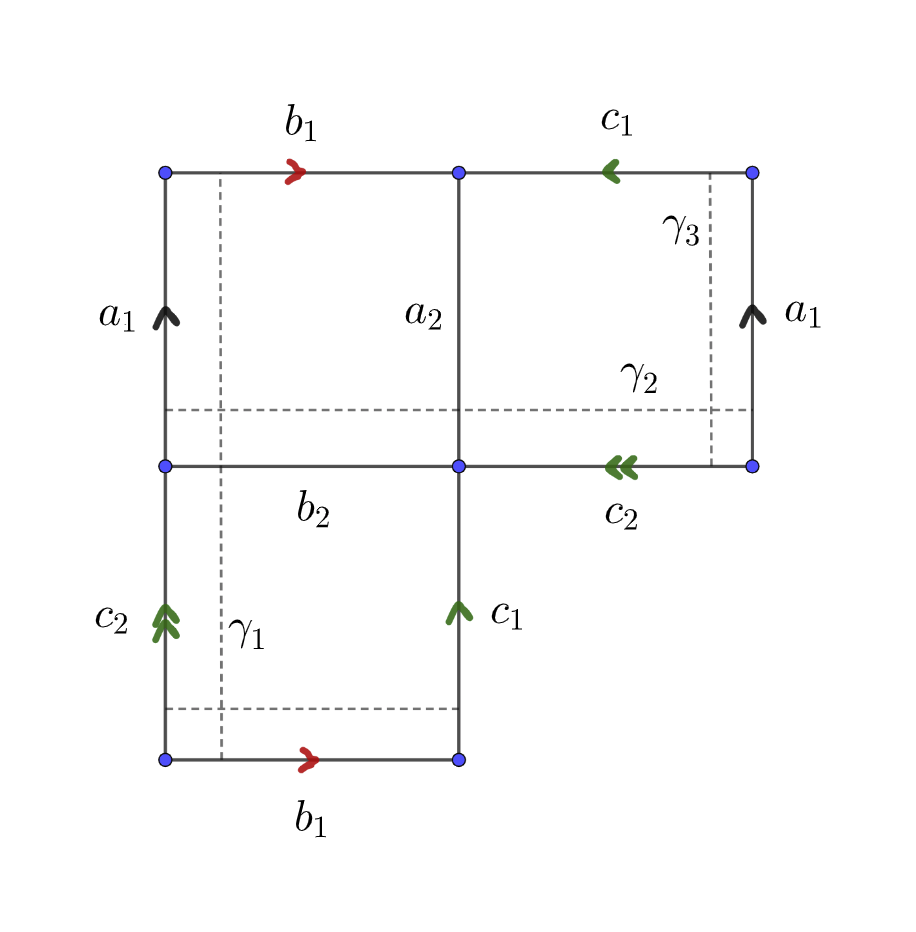}
   \caption{The non-orientable surface $N_4$}
    \label{N4}
\end{figure} The relation coming from the three squares above read: 
\begin{equation*}
    \begin{cases}
      a_1.b_1=b_2.a_2\\
      c_{2}^{-1}.a_1.c_1=a_2\\
     c_{2}^{-1}.b_1.c_1=b_2
     
    \end{cases}
\end{equation*}
For a representation $\rho \in \Hom(\pi_1(N_4),\SU(2))$, we write:   
\begin{equation*}
    \begin{cases}
      A_1.B_1=B_2.A_2\\
    C_{2}^{-1}.A_1.C_1=A_2\\
     C_{2}^{-1}.B_1.C_1=B_2
    
    \end{cases}
\end{equation*}

The group $\Gamma = \langle \tau_{\gamma_1},\tau_{\gamma_2},\tau_{\gamma_3} \rangle$ acts on $\Hom(\pi_1(N_4),\SU(2))$, and a function is $\Gamma$-invariant once it can be written in terms of each of the following 4-uplets $(A_1,A_2,B_1,B_2)$, $(A_1,A_2,C_1,C_2)$ and  $(B_1,B_2,C_1,C_2)$, simultaneously. 
Consider the isometry of $\mathbb{H}$, $\Phi_{C} : X \mapsto C_{2}^{-1}.X.C_1$. 
The system of equations can be rewritten as: 
\begin{equation*}
    \begin{cases}
      A_1.B_1=B_2.A_2\\
     \Phi_{C}(A_1)=A_2\\
     \Phi_{C}(B_1)=B_2
    \end{cases}
\end{equation*}
Taking the trace of the first equation we get: $$\tr(A_1.B_1)=\tr(A_2.B_2)$$
From the last two equations and since $\Phi_{C}$ is an isometry we deduce that the angle between $A_1$ and $B_1$ is equal to the angle between $A_2$ and $B_2$, in other words:
$$\tr(A_1.B_{1}^{-1})=\tr(A_2.B_{2}^{-1})$$
Adding the two previous equations and using the fundamental trace identity in $\SL(2,\mathbb{C})$ (i.e.  $\tr(XY)+\tr(XY^{-1})=\tr(X).\tr(Y)$, for any $X,Y \in \SL(2,\mathbb{C})$), we get: 
$$\tr(A_1)\tr(B_{1})=\tr(A_2)\tr(B_{2})$$
Finally, we deduce that the function: $$\frac{\tr(A_1)}{\tr(A_2)}=\frac{\tr(B_2)}{\tr(B_{1})}$$
is a $\Gamma$-invariant function on the $\SU(2)$-character variety of $N_4$. What is left to do is to check that the function is not constant, for this purpose, it is sufficient to prove that the pair $(A_1,A_2)$ can take any values in $\SU(2)^2$. 

\begin{lemma}
The projection $\rho \mapsto (A_1,A_2)$ from $\Hom(\pi_1(N_4),\SU(2))$ to $\SU(2)^2$ is surjective.      
\end{lemma}
\begin{proof}
 For $(A_1,A_2) \in \SU(2)^2$, take $B_1$ in the sphere $$\{X \in \SU(2) \mid \tr(X.A_{1}^{-1})=\tr(A_1.X.A_{2}^{-2})\}$$Therefore $B_2=A_1.B_1.A_{2}^{-1}$. The last condition is equivalent to saying that the angle between $A_1$ and $B_1$ is equal to the angle between $A_2$ and $B_2$. This ensures the existence of an isometry $\Phi_{C}$ in $\SO(4)$ such that $\Phi_{C}(A_1)=A_2$ and
     $\Phi_{C}(B_1)=B_2$.
\end{proof}

\end{document}